\documentclass[a4paper,11pt]{article}
\usepackage{amsmath}
\usepackage{amssymb}
\usepackage[pdftex]{graphicx}
\usepackage{pstricks}
\usepackage{enumerate}

\usepackage{exscale,relsize}

\usepackage{cite}
\usepackage{theorem} %
\RequirePackage{hyperref}
\PassOptionsToPackage{pdftex,bookmarksopen,bookmarksnumbered}{hyperref}
\usepackage{hyperref}

\newtheorem{thm}{Theorem}[section]
\newtheorem{lemma}[thm]{Lemma}
\newtheorem{cor}[thm]{Corollary}
\newtheorem{propn}[thm]{Proposition}
\newtheorem{defn}[thm]{Definition}

\theoremstyle{plain}{\theorembodyfont{\rmfamily}%

\theoremstyle{plain}{\theorembodyfont{\rmfamily}%
\newtheorem{algorithm}[thm]{Algorithm}
\theoremstyle{plain}{\theorembodyfont{\rmfamily}%

\theoremstyle{plain}{\theorembodyfont{\rmfamily}%
\newtheorem{remark}[thm]{Remark}

\numberwithin{equation}{section}

\newcommand{\Ucal}{\ensuremath{\mathcal U}}

\newcommand{\Ycal}{\ensuremath{\mathcal Y}}
\newcommand{\Xcal}{\ensuremath{\mathcal X}}

\renewcommand{\Bbb}{\ensuremath{\mathbb B}}
\newcommand{\Cbb}{\ensuremath{\mathbb C}}

\newcommand{\Nbb}{\ensuremath{\mathbb N}}

\newcommand{\Rbb}{\ensuremath{\mathbb R}}

\newcommand{\Nhat}{{\widehat{N}}}

\newcommand{\vbar}{{\overline{v}}}

\newcommand{\xbar}{{\overline{x}}}

\renewcommand{\equiv}{:=}

\newcommand{\attains}[1]{\rightarrow_{\!\!\!\!\!\!\!{_#1}}~}

\newcommand{\pinf}{\ensuremath{+\infty}}
\newcommand{\Ball}{{\Bbb}}

\newcommand{\Rn}{{\Rbb^n}}

\newcommand{\Rnn}{{\Rbb^{n\times n}}}
\newcommand{\Rmn}{{\Rbb^{m\times n}}}
\newcommand{\Rm}{{\Rbb^m}}

\newcommand{\ucmin}[2]{\ensuremath{\underset{\substack{{#2}}}%
{\mathrm{minimize}}\;\;#1 }}
\newcommand{\argucmin}[2]{\ensuremath{\underset{\substack{{#2}}}%
{\mathrm{argmin}}\;\;#1 }}

\newcommand{\set}[2]{\left\{#1\,\left|\,#2\right.\right\}}
\newcommand{\map}[3]{#1:\,#2\rightarrow #3\,}

\newcommand{\ip}[2]{\left\langle #1,~ #2\right\rangle}

\DeclareMathOperator{\range}{{\rm range}}
\DeclareMathOperator{\rank}{{\rm rank}}

\DeclareMathOperator{\cone}{{cone\,}}

\DeclareMathOperator{\supp}{supp\,} 
\DeclareMathOperator{\Supp}{Supp\,} 

\newcommand{\trace}[1]{{\bf Tr}\left(#1\right)\,}

\DeclareMathOperator{\vdiag}{diag\,}
\DeclareMathOperator{\mdiag}{Diag\,}

\definecolor{labelkey}{rgb}{0,0.08,0.45}
\definecolor{refkey}{rgb}{0,0.6,0.0}
\definecolor{Brown}{rgb}{0.45,0.0,0.05}
\definecolor{dgreen}{rgb}{0.00,0.49,0.00}
\definecolor{dblue}{rgb}{0,0.08,0.75}

\title{Prox-regularity of rank constraint sets\\ and implications for algorithms}

\author{
D. Russell Luke\\
Institut f\"ur Numerische und Angewandte Mathematik\\ 
Universit\"at G\"ottingen\\
37083, G\"ottingen, Germany.}
\date{\ttfamily Dec. 2, 2012 -- version 1.10.  To appear in the Journal of Mathematical Imaging and Vision.}

\begin{document}
\maketitle

\begin{abstract}
We present an analysis of sets of matrices with rank less than or equal to a specified number $s$.
We provide a simple formula for the normal cone to such sets, and use this to show that these sets
are prox-regular at all points with rank exactly equal 
to $s$.  The normal cone formula appears to be new.  
This allows for easy application of prior results guaranteeing local linear convergence of 
the fundamental alternating projection algorithm between sets, one of which is a rank constraint set.
We apply this to show local linear convergence of another fundamental algorithm, 
approximate steepest descent.  
Our results apply not only to linear systems with rank constraints, as has been treated extensively in the 
literature, but also nonconvex systems with rank constraints.  
\end{abstract}


\section{Introduction}  Rank optimization is a well-developed topic that has found
a tremendous number of applications in recent years (see \cite{RechtFazelParrilo10}
and references therein).  Most of the problems one encounters involve a linear data model
that is underdetermined and the very poorly behaved ``sparsity function'', either
the function determining the rank of a matrix or the function counting the number of nonzero
entries in an array.  A common approach to solving sparsity optimization problems
is via a convex surrogate, most often the $\ell_1$ or (in the case of matrices) the nuclear norm.
The rational for working with such surrogates is that the original problem is NP-complete, and 
thus should be avoided.  
Inspired by earlier work proving local convergence of cyclic projections onto nonconvex sets
with an application to sparse signal recovery \cite{CombettesTrussell90}, and a more recent  
projection-reflection algorithm for x-ray imaging \cite{Oszlanyi08} that appears to be very 
successful at working with a proximal operator of the $\ell_0$ function and a nonlinear 
imaging model, we set out in the present note to determine whether sets with sparsity
constraints have some sort of regularity that might justify working directly with sparsity
rather than through convex surrogates.  

Based on the work of Lewis and Sendov \cite{LewisSendov05, LewisSendov05b}, Le has 
obtained explicit formulas for the generalized rank function \cite{Le12}.  This formula shows 
that every point of the rank function is a critical point \cite{Hiriart-Urruty12}, and so reasonable algorithmic strategies
should not {\em directly} make use of the rank function.  Instead, we consider the {\em lower level sets}
of the rank function.  While sets of matrices of rank less than a specified level are not manifolds, we show
here that they are quite regular, in fact {\em prox-regular}.  While prox-regularity of these sets
is not new \cite{LewisMalick08}, our proof of this fact  established in Section \ref{s:geometry} 
uses elementary tools, at the center of which 
is a particularly simple and apparently new characterization of the {\em normal cone} to these sets 
established in Proposition \ref{t:normal cone of S}.  

Prox-regularity of the lower level sets of the rank function 
immediately yields local linear convergence of fundamental
algorithms for either finding the intersection of the rank constraint set with another set determined by some
(nonlinear) data model, or for minimizing the distance to a rank constrained set and a data set.  The 
result, detailed in Section \ref{s:algorithms}, is quite general and extends to nonconvex data 
imaging models with rank constraints.  Our results are an extension of results established 
recently in \cite{BLPWII} for the vector case, however at the cost of additional assumptions on the 
regularity of the solution set.  In particular, \cite{BLPWII} establishes local linear convergence, 
with radius of convergence,  of alternating projections between an affine constraint and the set of vectors with 
no more than $s$ nonzero elements {\em without any assumptions} on the regularity of the intersection 
of these sets, beyond the assumption that it is nonempty.  Our results, in contrast, are modeled after  
results of \cite{LewisMalick08} and \cite{LewisLukeMalick08} where a stronger regularity of the intersection 
is assumed.  We discuss the difficulties in extending the tools developed in \cite{BLPWI} to the matrix case 
in the conclusion.  In any case, avoiding convex surrogates is at the cost of global convergence guarantees:  
these results are local and offer no panacea for solving rank optimization problems.  Rather, this analysis shows
that certain {\em macro-regularity} assumptions such as restricted isometry or mutual coherence (see 
\cite{RechtFazelParrilo10} and references therein) play no role 
{\em asymptotically} in the convergence of algorithms, but rather have bearing only on the radius of convergence.  
We begin this note
with a review of notation and basic results and definitions upon which we build.
 
\section{Notation}\label{s:notation}
Throughout this paper $\Xcal$ and $\Ycal$  are Euclidean spaces.  
In particular we are interested in Euclidean spaces defined on $\Rmn$ where
we derive the norm from the trace inner product
\[
   \ip{y}{x}\equiv\trace{y^Tx}\quad\mbox{ for } x,y\in \Rmn, \quad \|x\|\equiv\sqrt{\trace{x^T x}}.
\]
This naturally specializes to the case of $\Rn$ when $m=n$ above and $x\in\Rnn$ is restricted
to the subspace of diagonal matrices.  
For $x\in\Rmn$ we denote the span of the rows of $x$ by $\range(x^T)$ and recall that 
this is orthogonal to the nullspace of the linear mapping $\map{x}{\Rn}{\Rm}$,
\[
   \range(x^T)=\ker(x)^\perp.
\]
For $x\in\set{z\in\Rnn}{z_{ij}=0~\mbox{ if }i\neq j}$ (that is, when $x$ is square diagonal)
this corresponds exactly to the usual support of vectors on $\Rn$:
\[
   \range(x^T)=\supp(\mdiag(x))\equiv\set{y\in\Rn}{y_i=0~\mbox{ for all }~ i\in\{1,2,\dots,n\}~\mbox{ with }~ \mdiag(x)_{i}=0}
\]
where $\mdiag(x)$ maps the diagonal of the matrix $x\in\Rmn$ to a vector in $\Rbb^r$ with $r=\min\{m,n\}$.   
In order to emphasize this connection to the support of vectors, and reduce notational clutter
we will denote the span of the rows of $x$ by 
\[
   \Supp(x)\equiv\range(x^T).
\]
 
We denote the rank of $x$ by  $\rank(x)$ and recall that $\rank(x)$ is the dimension of the 
span of the columns -- or equivalently the rows -- of $x$ which is equivalent to the number of nonzero singular values.  
The singular values of $x\in\Rmn$ are the (positive) square root of the eigenvalues of $xx^T$;  
these are denoted by $\sigma_j(x)$ and are assumed
to be ordered so that $\sigma_i(x)\geq \sigma_j(x)$ for $i<j$.  We denote by 
$\sigma(x)\equiv (\sigma_1(x),\sigma_2(x),\dots,\sigma_r(x))^T$ ($r=\min\{m,n\}$) the 
ordered vector of singular values of $x$.  The corresponding diagonal matrix is denoted
$\Sigma(x)\equiv\vdiag(\sigma(x))\in\Rmn$ where $\vdiag(\cdot)$ maps vectors in $\Rbb^r$ to matrices in $\Rmn$.  
Following \cite{Lewis06, LewisSendov05, LewisSendov05b} 
we denote the (Lie) group of $n\times n$ orthogonal 
matrices by $O(n)$ and the product $O(m)\times O(n)$ by $O(m,n)$.  
A singular value decomposition of $x\in \Rmn$ restricted to the above ordering is then any 
pair of orthogonal matrices $(U,V)\in O(m,n)$ together with $\Sigma(x)$ such that 
$x=U\Sigma(x)V^T$.  
We will denote the set of pairs of orthogonal matrices that 
comprise singular systems for $x$ by $\Ucal(x)\equiv\set{(U,V)\in O(m,n)}{x=U\Sigma(x)V^T}$.

The {\em closed} ball centered at $x$ with radius $\rho$ is denoted by $\Ball(x,\rho)$; the 
unit ball centered at the origin is simply denoted by $\Ball$.  
Given a set $\Omega\subset\Xcal$, we denote the distance 
of a point $x\in\Xcal$ to $\Omega$ by $d_\Omega(x)$ where 
\[
d_\Omega(x) \equiv\inf_{y\in\Omega}\|y-x\|.
\]
If $\Omega$ is empty then we use the convention that the distance to this set is 
$\pinf$.   
The corresponding (multivalued) 
projection operator of $x$ onto $\Omega$, denoted $P_\Omega(x)$, is defined by  
\[
P_\Omega(x)\equiv\argucmin{\|z-x\|}{z\in\Omega}.   
\]
If $\Omega$ is nonempty and closed, then the projection 
of any point in $\Xcal$ onto $\Omega$ is nonempty.  

We define the {\em normal cone}  to a closed set 
$\Omega\subset \Xcal$ following \cite[Def. 6.3]{VA}:
\begin{defn}[normal cone]
\label{d:normal cone} 
A vector $v\in \Xcal$ is {\em normal} to a closed set $\Omega\subset\Xcal$ at 
$\xbar\in \Omega$, written $v\in N_\Omega(\xbar)$
if there are sequences $(x^k)_{k\in\Nbb}$ in  $\Omega$ with $x^k\attains{\Omega}\xbar$ and 
$(v^k)_{k\in\Nbb}$ in  $X$ with $v^k\to v$ such that  
\[
\limsup_{\underset{x\neq x^k}{x\attains{\Omega}x^k}}\frac{\ip{v^k}{x-x^k}}{|x-x^k|}\leq 0.
\]
The vectors $v^k$ are {\em regular normals} to $\Omega$ at $x^k$ and 
the cone of regular normals at $x^k$ is denoted $\Nhat_\Omega(x^k)$.   
\end{defn}
What we are calling regular normals are called {\em Fr\'echet} normals in 
\cite[Def. 1.1]{Mor06}.

Here and elsewhere we use the notation $x\attains{\Omega}\xbar$ to mean that 
$x\to \xbar$ with $x\in\Omega$.  
An important example of a regular normal is a {\em proximal normal}, defined as 
any vector $v\in \Xcal$ that can be written as 
$v=\lambda(x-\xbar)$ for $\lambda\geq 0$ and $\xbar\in P_\Omega(x)$ for some 
$x\in \Xcal$.  We denote the set of proximal normals to $\Omega$ at $x\in \Omega$ 
by $N^P_\Omega(x)$.  
For $\Omega$ closed and nonempty, any normal $\vbar\in N_\Omega(\xbar)$
can be approximated arbitrarily closely by a proximal normal \cite[Exercise 6.18]{VA}. 
Thus we have the next result which 
is key to our analysis.
\begin{propn}[Theorem 1.6 of \cite{Mor06}]\label{t:N-PN}
Let $\Omega\subset\Xcal$ be closed and $\xbar\in\Omega$.  Then  
 \begin{eqnarray}\label{e:N-PN}
&&\!\!\!\!   \!\!\!\! \!\!\!\!\!\!\!\!N_\Omega(\xbar) = \nonumber\\
&&\!\!\!\! \!\!\!\!\!\!\!\!\set{v\in\Rn}{\exists\mbox{ sequences } x^k\to \xbar\mbox{ and }v^k\to v
\mbox{ with }v^k\in\cone(x^k-P_\Omega(x^k))\mbox{ for all }k\in\Nbb}.
\end{eqnarray}
\end{propn}

Central to our results is the {\em regularity} of the intersection of
sets, which we define in terms of a type constraint qualification formulated
with the normal cones to the sets at points in the intersection.   
\begin{defn}[basic set intersection qualification]  
\label{d:strong regularity}
A family of closed sets $\Omega_1$,$\Omega_2,\ldots$ $\Omega_m$ $\subset \Xcal$ 
satisfies the basic set intersection qualification at a point $\overline{x} \in \cap_i \Omega_i$, 
if the only solution to 
\[
\displaystyle{\sum_{i=1}^m} y_i  =  0,\quad 
y_i  \in  N_{\Omega_i}(\overline{x}) ~~ (i=1,2,\ldots,m)
\]
is $y_i = 0$ for $i=1,2,\ldots,m$.  We say that the intersection is {\em strongly regular} at $\overline{x}$
if the basic set constraint qualification is satisfied there.  
\end{defn}
In the case $m=2$, this condition can be written
\[
N_{\Omega_1}(\bar x) \cap - N_{\Omega_2}(\bar x) =\{0\}.
\] 
The two set case is called the {\em basic constraint qualification for sets} in \cite[Definition 3.2]{Mor06} and has its origins 
in the the {\em generalized property of nonseparability} \cite{Mord84} which is the $n$-set case.  It was later recovered as
a dual characterization of what is called {\em strong regularity} of the intersection in \cite[Theorem 3]{Kru05}.  
It is called {\em linear regularity} in \cite{LewisLukeMalick08}.

The case of two sets also yields the following simple quantitative characterization of strong regularity.

\begin{propn}[Theorem~5.16 of \cite{LewisLukeMalick08}]
\label{t:cbar}
Suppose that $\Omega_1$ and $\Omega_2$ are closed subsets of $\Xcal$.  The intersection 
$\Omega_1\cap \Omega_2$ satisfies the basic set intersection qualification at $\overline{x}$ if and only if the constant 
\begin{equation}\label{e:cbar}
\overline{c}  ~\equiv~ \sup \set{\ip{u}{v}}{u \in N_{\Omega_1}(\overline{x}) \cap \Ball,~ v \in -N_{\Omega_2}(\overline{x}) \cap \Ball}<1.
\end{equation}
\end{propn}

\begin{defn}[angle of regular intersections]
Suppose that $\Omega_1$ and $\Omega_2$ are closed subsets of $\Xcal$.
We say that the intersection $\Omega_1\cap \Omega_2$ 
is  {\em strongly regular at $\overline{x}\in \Omega_1\cap \Omega_2$ with angle
$\overline{\theta}\equiv\cos^{-1}(\overline{c})>0$}
when the constant $\overline{c}$ given by \eqref{e:cbar} is less than $1$.
\end{defn}

We will also require certain regularity of the sets themselves, not
just the intersection.  The 
following definition of {\em prox-regularity} of sets is a modern manifestation that can be traced
back to \cite{Federer59} and {\em sets of positive reach}.  What we use here as a definition
actually follows from the equivalence of prox-regularity of sets as defined in 
\cite[Definition 1.1]{PolRockThib00} and the single-valuedness of the projection
operator on neighborhoods of the set \cite[Theorem 1.3]{PolRockThib00}.
\begin{defn}(prox-regularity)
A nonempty closed set $\Omega\subset\Xcal$ is {\em prox-regular} at a point
$\overline{x}\in \Omega$ if $P_C(x)$ is single-valued around $\xbar$.
\end{defn}

\section{Properties of lower level sets of the rank function}\label{s:geometry}
We collect here some facts that will be used repeatedly in what follows.  

\begin{propn}\label{t:l0 of converging sequences}
For any point $\overline{x}\in \mathbb{R}^{m\times n}$ and any sequence 
$(x^k)_{k\in\mathbb{N}}$ converging to $\overline{x}$ there is a $K\in \mathbb{N}$ such that
$\rank(\overline{x})\leq \rank(x^k)$ for all $k>K$.
\end{propn}

\begin{proof}
   This follows immediately from continuity of the singular values as a function of $x$.  
(See, for instance, \cite[Appendix D]{HornJohnson1}.)
\end{proof}

For the remainder of this note we will consider real $m\times n$ matrices and denote by $r$ the minimum of $\{m,n\}$.  The 
rank level set will be denoted by $S:=\left\{y\in\mathbb{R}^{m\times n}~\left|~\rank(y)\leq s\right.\right\}$ for $s\in\{0,1,\dots,r\}$.  
As can be found in textbooks on matrix analysis, the projection onto this set is just the truncation 
of the $r-s$ smallest singular vectors to zero; in the 
case of a tie for the $s$-th largest singular value, the projection is the set of all $s$-selections from the 
$s$-largest singular values.  
\begin{lemma}[projection onto S]\label{t:P_S}
For $x\in\mathbb{R}^{m\times n}$, 
define 
\[
  \Sigma_s(x):=\vdiag((\sigma_1(x),\sigma_2(x),\dots,\sigma_s(x),0,\dots,0)^T)\in \mathbb{R}^{m\times n}. 
\]
The projection $P_S(x)$ is given by 
\[
P_S(x)=\bigcup_{(U,V)\in\mathcal{U}(x)}\left\{y~\left|~y=U\Sigma_s(x)V^T\right.\right\}. 
\]
\end{lemma}
\begin{proof}
By \cite[Theorem 7.4.51]{HornJohnson1} any matrix $y\in S$ 
satisfies $\|x-y\|\geq\|\Sigma(x)-\Sigma(y)\|$.  The relation holds with equality whenever 
$y=U\Sigma_s(x)V^T$ for some $(U,V)\in \mathcal{U}(x)$, hence 
$P_S(x)\supset\bigcup_{(U,V)\in\mathcal{U}(x)}\left\{y~\left|~y=U\Sigma_s(x)V^T\right.\right\}\neq\emptyset.$  On the other hand, 
if $\overline{y}\in P_S(x)$, then $\|x-\overline{y}\|\leq\|x-y\|$ for all $y\in S$.  In particular, for $y=U\Sigma_s(x)V^T$ with
$(U,V)\in\mathcal{U}(x)$ we have
\begin{eqnarray*}
   \|x-\overline{y}\|\leq\|x-y\| = \|\Sigma(x)-\Sigma_s(x)\|\leq \|\Sigma(x)-\Sigma_s(\overline{y})\|\leq \|x-\overline{y}\|,
\end{eqnarray*}
hence  $\Sigma_s(x)=\Sigma(\overline{y})$ and $\overline{y}\in \bigcup_{(U,V)\in\mathcal{U}(x)}\left\{y~\left|~y=U\Sigma_s(x)V^T\right.\right\}.$
\end{proof}

The next results establish that the set $S$ is prox-regular at all points where $\rank(x)=s$.  We make use of the 
following tools.  For $r=\min\{m,n\}$ 
define the mappings $\mathbb{J}:~\mathbb{R}^{m\times n}\times(\mathbb{R}_+\cup \{+\infty\})\to 2^{\{1,2,\dots,r\}}$ and 
$\alpha_s(x):~\mathbb{R}^{m\times n}\to [0,+\infty]$ by 
\[
\mathbb{J}(x,\alpha):=\left\{j\in\{1,2,\dots,r\}~\left|~\sigma_j(x)\geq \alpha\right.\right\}   
\mbox{ and } 
\alpha_s(x):=\sup\left\{\alpha~\left|~|\mathbb{J}(x,\alpha)|\geq s\right.\right\},
\]
where $|\mathbb{J}(x,\alpha)|$ denotes the cardinality of this discrete set.  
We define $\mathbb{J}(x,+\infty)$ to be the empty set.
Before proceeding with our results, we collect some observations about these
objects.
\begin{lemma}\label{t:Jbb}$~$
\begin{enumerate}[{(i)}]
\item\label{t:Jbb_i} For all $s\in \{1,2,\dots,r\}$ the value of the supremum in the definition of $\alpha_s(x)$ is bounded and attained.   
If $s=0$ then $\alpha_0(x)=+\infty$.
\item\label{t:Jbb_ii} If  $\left|\mathbb{J}(x,\alpha_s)\right|>s$ then $\rank(x)>s>0$.
\item\label{t:Jbb_iii} If  $\rank(x)>s$ then $\alpha_s>0$.
\item\label{t:Jbb_iv} If $\rank(x)<s\leq r$ then $\alpha_s(x)=0$.
\end{enumerate}
\end{lemma}
\begin{proof}
(i) Since the cardinality of the empty set is zero, the supremum in the 
definition of $\alpha_0$ is unbounded.   In any case, the cardinality of 
$\mathbb{J}(x,\alpha)$ is monotonically decreasing with respect to $\alpha$ for $x$ fixed 
from a value of $r$ at $\alpha=0$ to $0$ for all $\alpha>\sigma_1(x)$.  Thus for $x$ fixed 
$\alpha_s$ is bounded for all $s\in \{1,2,\dots, r\}$.  The value of $\alpha$ 
for which the cardinality $s\geq 1$ is achieved is attained precisely when $\alpha=\sigma_j(x)$ for some $j$.   
(ii) By definition, at $s=0$, $\alpha_0=+\infty$ and $|\mathbb{J}(x,+\infty)|:= 0$, so $\left|\mathbb{J}(x,\alpha_s)\right|>s$ 
implies that  $s>0$ and the implication $\left|\mathbb{J}(x,\alpha_s)\right|>\rank(x)$ follows immediately.   
(iii) If $\rank(x)>s$ and $s=0$, then the result is trivial since $\alpha_0:=+\infty$. 
If $\rank(x)>s$ and $s>0$ then $s\in \{1,\dots,r-1\}$ 
(it is impossible to have rank greater than $r$) 
and there exists an $\alpha>0$ such that $|\mathbb{J}(x,\alpha)|\geq s+1$. As
$\alpha_{s+1}$ is the maximum of these, $\alpha_{s+1}>0$.   
By the argument in (i) $\alpha_s\geq \alpha_{s+1}$ which yields the result. 
(iv) In this case, only by including the zero singular values of $x$ can the inequality $|\mathbb{J}(x,\alpha)|\geq s$ be 
achieved, that is by taking $\alpha=0$.  
\end{proof}

\begin{propn}[properties of the projection]\label{t:P_S properties}
The following are equivalent.
\begin{enumerate}[{(i)}]
   \item  $P_S(x)$ is multi-valued;
\item $\left|\mathbb{J}(x,\alpha_s)\right|>s$.
\end{enumerate}
\end{propn}
\begin{proof}
To show that $(i)$ implies $(ii)$, let $y$ and $z\in P_S(x)$ with $y\neq z$.  By Lemma \ref{t:P_S}
$y=U_y\Sigma_s(x)V^T_y$ and $z=U_z\Sigma_s(x)V^T_z$ for $(U_y,V_y)$ and 
$(U_z,V_z)\in \mathcal{U}(x)$.  Then by \cite[Theorem 7.4.51]{HornJohnson1}
\[
   0=\|\Sigma(y)-\Sigma(z)\|<\|y-z\|\leq \|y-x\|+\|x-z\|\\
 = 2\|\Sigma(x)-\Sigma_s(x)\|
\]
hence $\rank(x)>s$.  Since $y\neq z$ and they have the same singular values, the multiplicity of 
singular values $\sigma_j(x)$ with value $\alpha_s$ must be greater than one, hence 
$\left|\mathbb{J}(x,\alpha_s)\right|>s$.

Conversely, to show that $(ii)$ implies $(i)$ first note that by Lemma \ref{t:Jbb}(\ref{t:Jbb_ii}) $\rank(x)>s>0$.   Now
fix $y\in P_S(x)$ with $y=U_y\Sigma_s(x)V^T_y$ for $(U_y,V_y)\in \mathcal{U}(x)$.  
The corresponding decomposition for $x$ is  $U_y\Sigma(x)V^T_y$.  Now construct the orthogonal matrix $\widetilde{V}$ by 
switching the $s+1'$th column of $V_y$ with the  
$s'$th column of $V_y$.  Since   $\left|\mathbb{J}(x,\alpha_s)\right|>s$
we have that $x=U_y\Sigma(x)\widetilde{V}^T$. 
Define $z:= U_y\Sigma_s(x)\widetilde{V}^T$.  By Lemma \ref{t:P_S} $z\in P_S(x)$ with $\rank(z)=s$, but the $s'$th column 
of $\widetilde{V}$ is in the nullspace of $y$ so $z\neq y$ and the projection is thus multi-valued.  This completes the proof.
\end{proof}

An immediate consequence of the above is the obvious observation that the projection onto the trivial  sparsity sets 
$S_0$ with $s=0$ and $S_r$ with $s=r$ is  single-valued.  
\begin{cor}
For $x\in \mathbb{R}^{m\times n}$, if $s=0$ or $s=r$ then $P_S(x)$ is single-valued.   
\end{cor}

The normal cone of this set has the 
following simple characterization.
\begin{propn}[the normal cone to S]
\label{t:normal cone of S}
At a point $\overline{x}\in S$
\begin{equation}\label{e:N_S}
N_{S}(\overline{x}) =
\left\{v\in\mathbb{R}^{m\times n}~\left|~\ker(v)^\perp\cap\ker(\overline{x})^\perp=\{0\}\mbox{ and } \rank(v)\leq r-s\right.\right\}.   
\end{equation}
Moreover, $N_S(\overline{x})=N^P_S(\overline{x})$  at every $\overline{x}$ with $\rank(\overline{x})=s$, 
while $N^P_S(\overline{x})=\{0\}$ at every $\overline{x}$ with $\rank(\overline{x})<s$.   
\end{propn}

\begin{proof}  Using the definition of $\Supp(x):=\ker(x)^\perp$  
define the sets \\
\begin{eqnarray*}
W&:=&\left\{v\in\mathbb{R}^{m\times n}~\left|~\Supp(v)\cap\Supp(\overline{x})=\{0\}\mbox{ and }  \rank(v)\leq r-s\right.\right\}   \\
\mbox{ and }\quad Z(w)&:=&
\left\{z\in\mathbb{R}^{m\times n}~\left|~\Supp(z)\cap\Supp(w)=\{0\}\mbox{ and }  \rank(\overline{x}+z)=s\right.\right\}.
\end{eqnarray*}
We first show that $W$ is nonempty and hence $Z(w)$ for $w\in W$ is nonempty.  For all $\overline{x}\in \mathbb{R}^{m\times n}$ and 
$s\in \{0,1,2,\dots,r\}$ the 
zero matrix $0\in W$, hence $W$ is nonempty.  Next note that for $w\in W$, $Z(w)\subset \ker(w)$ 
with $\dim(\ker(w))\geq s\geq 0$, and 
it is always possible to find an element $z$ of $\ker(w)$ with $\rank(\overline{x}+z)=s$. 

Now, choose any $w\in W$ and $z_0\in Z(w)$ and  
construct the sequences $(x^k)_{k\in\mathbb{N}}$ and $(w^k)_{k\in\mathbb{N}}$ by
\[
x^k = \overline{x} +\tfrac1k w + \frac{1}{\sqrt{k}} z_0 \quad\mbox{ and }\quad 
w^k=k\left(x^k-y^k \right),~ \text{for} ~y^k\in P_{S}(x^k)~~ (k\in \mathbb{N}).   
\]
  There is a $K\in \mathbb{N}$ such that for all $k>K$ 
\[
   \frac1k\max_j\{\sigma_j(w)\}<\min_j\left\{\sigma_j\left(\frac{1}{\sqrt{k}}z_0+\overline{x}\right)~\left|~%
\sigma_j\left(\frac{1}{\sqrt{k}}z_0+\overline{x}\right)\neq 0\right.\right\}.
\]
Thus for all $k>K$ 
\[
   P_S(x^k) = \overline{x} +\tfrac{1}{\sqrt{k}}z_0\quad \mbox{ and }\quad  w^k = k\left(x^k - \left(\overline{x} +\tfrac{1}{\sqrt{k}}z_0\right)  \right)=w.
\]
Note that by Proposition \ref{t:P_S properties} and Lemma \ref{t:Jbb}(\ref{t:Jbb_ii}) the representation of the projection above holds with equality since 
$\rank\left(\overline{x} +\tfrac{1}{\sqrt{k}}w_0\right)=s$.  Since
$x^k\to \overline{x}$, by definition, $w\in N_S(\overline{x})$.  As $w$ was arbitrary, we have $W\subset N_S(\overline{x})$.   

We show next that, conversely, $N_S(\overline{x})\subset W$ for $\overline{x}\in S$. 
The matrix $w=0$ trivially belongs to $W$, so we assume that $w\neq 0$.  
By Proposition \ref{t:N-PN} we can write $w$ as a limit of proximal normals, that is, the limit  
of  sequences $(x^k)$ and $(w^k)$ with $x^k\notin S$ and $w^k\to w$ for 
$w^k = t^k\left(x^k-y^k\right) \text{ for } y^k\in P_S(x^k)$.  
We consider the corresponding singular value decompositions  
by $y^k=U_k\Sigma_s(x^k)V_k^T$ for $(U_k,V_k)\in \mathcal{U}(x^k)$  and 
$\Sigma_s(x):=\vdiag((\sigma_1(x),\sigma_2(x),\dots,\sigma_s(x),0,\dots,0)^T)\in \mathbb{R}^{m\times n}$
(see Lemma \ref{t:P_S}).   
Note that $x^k$ and $y^k$ 
have the same left and right singular vectors with the usual ordering.  The matrices $U_k$ and 
$V_k$ are also collections of left and right singular vectors for $w^k$, although they do not yield 
the usual ordering of singular values of $w^k$: 
\[
  w^k=t_kU_k\widetilde{\Sigma}_s(x^k)V_k^T\quad\mbox{ for } \quad
\widetilde{\Sigma}_s(x^k):=\vdiag((0,0,\dots,0,\sigma_{s+1}(x^k), \sigma_{s+2}(x^k),\dots,\sigma_{r}(x^k))^T)
\]
   Let $(\overline{U}, \overline{V})\in\mathcal{U}(\overline{x})$ be the limit of left and right singular vectors of $x^k$, that is,
$U_k\to \overline{U}$, $V_k\to \overline{V}$ where $x^k=U_k\Sigma(x^k)V_k\to \overline{U}\Sigma(\overline{x})\overline{V}=\overline{x}$.
Then $y^k=U_k\Sigma_s(x^k)V_k\to \overline{U}\Sigma(\overline{x})\overline{V}$ and 
$w^k=t_k U_k\widetilde{\Sigma}_s(x^k)V_k\to \overline{U}\left(\lim_{k\to\infty}t_k\widetilde{\Sigma}_s(x^k)\right)\overline{V}=w$.  It 
follows immediately that $\rank(w)\leq r-s$ and  $\Supp(w)\perp\Supp(\overline{x})$ which completes the proof of 
the inclusion.

To see that each normal to the set $S$ at $\overline{x}$ with $\rank(\overline{x})=s$  is actually a proximal normal, 
note that 
if $\rank(\overline{x})=s$ then by \eqref{e:N_S}  every point $v\in N_S(\overline{x})$ can be 
written as $v=\frac{1}{\tau}\left((\tau v+\overline{x})-P_S(\tau v+\overline{x})\right)$ for $\tau>0$ small enough.  
Suppose, on the other hand, that $\rank(\overline{x})<s$.  Then  $P_S(\tau v+\overline{x}) = \overline{x}$ for $\tau>0$ 
exactly when $v=0$:  for if  $\tau v+\overline{x}\in S$ then $P_S(\tau v+\overline{x}) = \tau v+\overline{x}=\overline{x}$ exactly when
$v=0$, and if 
$\tau v+\overline{x}\notin S$ then $\rank(P_S(\tau v+\overline{x}))=s$ hence $P_S(\tau v+\overline{x})\neq \overline{x}$.  
Consequently the only proximal normal at these points is $v=0$.   This completes the proof.
\end{proof}

The normal cone condition $N_{S}(\overline{x})\cap(-N_\Omega(\overline{x}))=\{0\}$ 
can easily be checked by determining the nullspace of matrices in $N_\Omega(\overline{x})$
as the next theorem shows.
\begin{propn}[strong regularity of intersections with a sparsity set]\label{t:strong reg S}
Let $\Omega\subset\mathbb{R}^{m\times n}$ be closed.   If at a point $\overline{x}\in\Omega\cap S$ all nonzero 
$v\in N_\Omega(\overline{x})$ have $\ker(v)^\perp\cap\ker(\overline{x})^\perp\neq \{0\}$, then 
the intersection is strongly regular there.  
\end{propn}
\begin{proof}
   Choose any $v\in N_\Omega(\overline{x})$.  Since $\Supp(v)\cap\Supp(\overline{x})\neq \{0\}$ and 
$N_{S}(\overline{x})$ given by \eqref{e:N_S} is a subset of matrices $w$ with 
$\Supp(w)\cap\Supp(\overline{x})=\{0\}$ the only solution to 
$v-w=0$ is $v=w=0$.  
\end{proof}

It is known that the set of matrices with rank $s$ is a smooth manifold \cite{Arnold71} (although 
the set of matrices with rank less than or equal to $s$ is not), from 
which it follows that $S$ is prox-regular \cite[Lemma 2.1 and Example 2.3]{LewisMalick08}. 
We present here a simple proof of this fact based on the characterization of the normal cone.
\begin{propn}[prox-regularity of S]\label{t:prox-reg S}
  The set $S$ is prox-regular at all points $\overline{x}\in \mathbb{R}^{m\times n}$ with $\rank(\overline{x})=s$.  
\end{propn}
\begin{proof}
Let $(x^k)_{k\in\mathbb{N}}$ in $\mathbb{R}^{m\times n}$ be any sequence converging to $\overline{x}$ with the 
corresponding singular value decomposition $U_k\Sigma(x^k)V_k^T$.  Decompose $x^k$ 
into the sum  $y^k+z^k=x^k$
where $y^k=U_k\Sigma_s(x)V_k^T$ and $z^k=U_k\widetilde{\Sigma}_s(x^k)V_k^T$ with 
$\Sigma_s(x^k):=\left(\sigma_1(x^k),\sigma_2(x^k),\dots, \sigma_s(x^k),0\dots,0\right)^T$ and 
$\widetilde{\Sigma}_s(x^k):=\left(0,\dots,0,\sigma_{s+1}(x^k),\sigma_{s+2}(x^k),\dots, \sigma_r(x^k)\right)^T$
for $r=\min\{m,n\}$. 
Note that $y^k\to \overline{x}$ with $\rank(y^k)=\rank(\overline{x})=s$ for all $k$ large enough, while
by Proposition \ref{t:normal cone of S}
$z^k\to 0$ with $z^k\in N_S(y^k)$  for all $k$.  Then 
for all $k$ large enough $\max_j\{\sigma_j(z^k)\}=\sigma_{s+1}(x^k)<\sigma_{s}(x^k)=\min_j\{\sigma_j(y^k)\}$ and $|\mathbb{J}(x^k,\alpha_s)|=s$.  By 
Proposition \ref{t:P_S properties} the projection $P_S(x^k)$ is single-valued. 
Since the sequence was arbitrarily chosen, it follows that the projection is single-valued on a 
neighborhood of $\overline{x}$, hence $S$ is prox-regular.  
\end{proof}

\section{Algorithms for optimization with a rank constraint}\label{s:algorithms}
The prox-regularity of the set $S$ has a number of important implications regarding 
numerical algorithms.  Principal among these is local linear convergence of the elementary alternating 
projection and steepest descent algorithms.  There has been a tremendous number of articles 
published in recent years about convex (and nonconvex) relaxations of the rank function, and when 
the solution of optimization problems with respect to these relaxations corresponds to the optimization
problem with the rank function (see the review article \cite{RechtFazelParrilo10} and references therein).  
The motivation for such relaxations is that there are polynomial-time 
algorithms for the solution of the relaxed problems, while the rank minimization problem is NP-complete.   
As we will show in this section, the above theory implies that in the neighborhood of a solution there are 
polynomial-time algorithms for the solution of optimization problems with rank constraints.  
This observation was anticipated in \cite{AttouchBolteRedontSoubeyran} and notably \cite{BeckTeboulle11} 
where a (globally) linearly convergent projected 
gradient algorithm with a rank constraint was presented.  Without further assumptions, however, 
such assurances of convergence of algorithms 
for problems with rank constraints is at the cost of global guarantees of convergence.  

\subsection{Inexact, extrapolated alternating projections}     
To the extent that the singular value decomposition can be computed exactly, the projection of a point $x$ 
onto the rank lower level set $S$ can be calculated exactly simply by ordering the singular values of $x$ and truncating.  
The above analysis immediately yields local linear convergence of exact
and inexact alternating  projections for finding the intersection $S\cap M$ for $M$ closed on neighborhoods of 
points where the intersection is strongly regular.  The following algorithm allows for inexact evaluation of the 
fixed point operator, and hence implementable algorithms.  

\begin{algorithm}[inexact alternating projections \cite{Luke11}]\label{alg:inexact ap}
\begin{description}
\item Fix $\gamma>0$ and choose $x^0 \in S$ and $x^1 \in M$.  
For $k=1,2,3,\ldots$ 
generate the sequence $(x^{2k})$ in $S$ with
$
x^{2k} \in P_S(x^{2k-1})
$
where the sequence $(x^{2k+1})$ in $M$ satisfies
\begin{subequations}\label{e:pooping}
\begin{eqnarray}
&& \|x^{2k+1} - x^{2k}\| \le \|x^{2k}-x^{2k-1}\|,\label{e:pooping_a}\\
&&x^{2k+1} = x^{2k}\quad\text{ if  }~x_*^{2k+1}=x^{2k},\label{e:pooping_b} \\
\text{ and }&& d_{N_M(x_*^{2k+1})}(\hat{z}^k)\le \gamma\label{e:pooping_c}
\end{eqnarray}
\end{subequations}
for 
\[
x_*^{2k+1}= P_{M\cap\{x^{2k}-\tau\hat{z}^k,~\tau\ge 0\}}(x^{2k})
\]
and
\[
\hat{z}^k\equiv \begin{cases}
 \frac{x^{2k} - x^{2k+1}}{\|x^{2k} - x^{2k+1}\|} &\text{ if }~x_*^{2k+1}\neq x^{2k}\\
0&\text{ if }~ x_*^{2k+1}=x^{2k}.
\end{cases}
\]
\end{description}
\end{algorithm}
For $\gamma=0$ and $x^{2k+1}= x^{2k+1}_*$ the inexact algorithm reduces to the usual 
alternating projections algorithm.  Note that the odd iterates $x^{2k+1}$ can lie on the interior of $M$.  
This is the major difference between Algorithm \ref{alg:inexact ap} and the one specified in 
\cite{LewisLukeMalick08} where all of the iterates are assumed to lie on the boundary
of $M$.  We include this feature to allow for {\em extrapolated} iterates in the case 
where $M$ has interior.  
  
\begin{thm}[inexact alternating projections with a rank lower level set]
\label{t:approx proj}
Let\\ $M,S \subset \Rmn$ be closed with $S\equiv\set{y}{\rank(y)\leq s}$ and  suppose there is 
an $\overline{x}\in M\cap S$ with $\rank(\overline{x})=s$.  Suppose furthermore that 
$M$ and $S$ have strongly 
regular intersection at $\overline{x}$ with angle $\overline{\theta}$.  Define
$\overline{c}\equiv\cos(\overline{\theta})<1$ and   
fix the constants $c \in (\overline{c},1)$  
and $\gamma < \sqrt{1-c^2}$. 
For $x^0$ and $x^1$ close enough to $\overline{x}$, the iterates in Algorithm
\ref{alg:inexact ap} converge to a point in
$M \cap S$ with R-linear rate
\[
\sqrt{ c\sqrt{1-\gamma^2} + \gamma \sqrt{1-c^2} } ~<~ 1.
\]
If, in addition, $M$ is prox-regular at $\overline{x}$, then  
the iterates converge with rate
\[
c\sqrt{1-\gamma^2} + \gamma \sqrt{1-c^2}  ~<~ 1.
\]
\end{thm}
\begin{proof}
   Since by Proposition \ref{t:prox-reg S} $S$ is prox regular at $\overline{x}$ the results follow 
immediately from \cite[Theorem 4.4]{Luke11}.
\end{proof}
\begin{remark}
   The above result requires only closedness of the set $M$.  For example, this yields convergence for 
affine sets $M=\set{x}{Ax=b}$ which are not only closed, but convex.  But the above result is not restricted 
to such nice sets.  Another important example is inverse scattering with 
sparsity constraints \cite{Oszlanyi08}.  Here the set $M$ is $M=\set{x\in\Cbb^n}{|(Fx)_j|^2=b_j,~j=1,2,\dots,n}$ where 
$F$ is a linear mapping (the discrete Fourier or Fresnel transform) and $b$ is some measurement
(a far field intensity measurement).  This set is not convex, but it is certainly closed (in fact prox-regular), 
so again, we can apply the above 
results to provide local guarantees of convergence for {\em nonconvex}
alternating projections with a sparsity set.  
\end{remark}
Paradoxically, in the vector case it is the projection onto the affine constraint that in general cannot  
be evaluated  exactly, while the projection onto the sparsity set $S$ {\em can } be implemented exactly 
 by simply (hard) thresholding the vectors.  In the matrix case, this is no longer possible since
in general the singular values cannot be evaluated exactly.  In order to accommodate both projections being 
approximately evaluated, we explore one possible solution using a common reformulation of the problem 
on a product space.  This is explained next.  

\subsection{Approximate steepest descent}
Another fundamental approach to solving such problems is simply to minimize the 
sum of the (squared) distances to the sets $M$ and $S$:
\[
   \ucmin{\frac12\left(d^2(x,S)+d^2(x,M)\right)}{x\in\Rmn}
\]
 Steepest descent without line search is: given $x_0\in\Rmn$ 
generate the sequence $(x^k)_{k\in\mathbb{N}}$ in $\Rmn$ via
\[
   x^{k+1}=x^k-\nabla\frac12\left(d^2(x^k,S)+d^2(x^k,M)\right).
\]
If $S$ and $M$ were convex and the distance function the Euclidean distance,  
it is well-known that this would be equivalent to 
averaged projections:
\begin{equation}\label{e:nabla-P}
   x^{k+1}=x^k-\nabla\frac12\left(d^2(x^k,S)+d^2(x^k,M)\right) = \frac12\left(P_S(x^k)+P_M(x^k)\right).
\end{equation}
If we assume that $M$ is prox-regular, then, since we have already established the 
prox-regularity of $S$, the correspondence between the derivative of the sum of 
squared distances to these sets and the projection operators in \eqref{e:nabla-P} 
holds on (common) open neighborhoods of $M$ and $S$ \cite[Theorem 1.3]{PolRockThib00}.  
Using a common product space formulation due to \cite{Pierra76} we can show that 
\eqref{e:nabla-P} is equivalent to alternating projections between the sets 
\[
   D\equiv\set{(x,y)\in \Rmn\times\Rmn}{x=y}
\]
and 
\[
\Omega\equiv\set{(x,y)\in\Rmn\times\Rmn}{x\in S, y\in M},
\]
that is, 
\[
   (x^{k+1},x^{k+1}) = P_D(P_\Omega((x^k,x^k)))
\]
where $x^{k+1}$ is given by \eqref{e:nabla-P}.  The set $\Omega$ is prox-regular if 
$M$ and $S$ are, and the set $D$ is convex, so Theorem \ref{t:approx proj} guarantees
local linear convergence of the sequence of iterates $(x^k)_{k\in\mathbb{N}}$ with rate depending on the 
angle of strong intersection of the sets $D$ and $\Omega$.  We cannot expect to be able to compute the 
projection onto the set $\Omega$ exactly, but we can reasonably assume to be able to compute the 
projection onto the diagonal $D$ exactly, even if the magnitudes of the elements of $P_S(x^k)$ and 
$P_M(x^k)$ differ in orders of 
magnitude beyond our numerical precision.  Indeed, since the projection operators $P_S$ and $P_M$ are
Lipschitz continuous for $S$ and $M$ prox-regular \cite[Theorem 1.3]{PolRockThib00}, we can attribute 
any error we in fact make in the evaluation of 
$P_D$ to the evaluation of $P_\Omega$ where we compute an approximation according to 
Algorithm \ref{alg:inexact ap}.  Again, Theorem \ref{t:approx proj} guarantees local linear convergence with 
rate governed by the angle of strong regularity between $D$ and $\Omega$ 
and the accuracy of the approximate projection onto
$\Omega$. 

\section{Conclusion}
We have developed a novel characterization of the normal cone to the lower level sets of 
the rank function.  This enables us to obtain a simple proof of the prox-regularity of such 
sets.  This property then allows for a straight-forward 
application of previous results on the local linear convergence of approximate alternating projections for 
finding the intersection of rank constrained sets and another closed set, as long as the intersection 
is {\em strongly regular} at a reference point $\overline{x}$.  
Our characterization of the normal cone to rank constraint sets allows
for easy characterization and verification of the strong regularity of intersections of these sets with other sets.  
The results are also extended to the elementary steepest descent algorithm for minimizing the 
sum of squared distances to sets, one of which is a rank constraint set.  This implies that, in the 
neighborhood of a solution with sufficient regularity, there are polynomial time algorithms for directly solving rank constraint 
problems without resorting to convex relaxations or heuristics.  

What remains to be determined is the radius of convergence of these algorithms.  
Using the {\em restricted} normal cone developed in \cite{BLPWI} Bauschke an coauthors 
\cite{BLPWII}
obtained linear rates with estimates of the radius of convergence of alternating projections 
applied to {\em affine} sparsity constrained 
problems -- that is, the vector affine case of the setting considered here -- assuming only 
existence of solutions.  The restricted normal cone is not immediately applicable here
since the restrictions in \cite{BLPWII} are over countable collections of subspaces representing 
all possible $s-$sparse vectors.  For the rank function this is problematic since the collection of 
all possible $s$-rank matrices is not countable.  Extending the tools of \cite{BLPWI} to the 
matrix case is the focus of future research.   

The results one might obtain 
using the tools of \cite{BLPWI} or similar, however,  are based on the regularity near the solution, 
what we call {\em micro-}regularity.  We cannot expect the estimates for the radius of 
convergence to extend very far using these tools, unless certain local-to-global properties like 
convexity are assumed.   In \cite{BeckTeboulle11}
a scalable restricted isometry property is used to prove global convergence of a projected gradient algorithm
to the unique solution to the problem of minimizing the distance to an affine subspace subject to a rank constraint.  
The (scalable) restricted isometry property and other properties like it (mutual coherence, etc) directly concern  
uniqueness of solutions and indirectly provide sufficient conditions
for global convergence of algorithms for solving relaxations of the original sparsity/rank optimization problem.  
A natural question is whether there is a more general {\em macro-regularity} property than the scalable restricted 
isometry property, one independent of considerations of uniqueness of solutions, that guarantees global 
convergence.  

\section*{Acknowledgments}The author would like to thank Patrick Combettes, Adrian Lewis and Arne Rohlfing 
for helpful feedback during the preparation of this work. 

\begin{thebibliography}{10}

\bibitem{Arnold71}
V.I. Arnol'd.
\newblock {On matrices depending on parameters.}
\newblock {\em {translation from Usp. Mat. Nauk}}, { 26}({2(158)}):{101--114},
  1971.

\bibitem{AttouchBolteRedontSoubeyran}
H.~Attouch, J.~Bolte, P.~Redont, and A.~Soubeyran.
\newblock Proximal alternating minimization and projection methods for
  nonconvex problems: An approach based on the {K}urdyka-{L}ojasiewicz
  inequality.
\newblock {\em Mathematics of Operations Research}, 35(2):438--457, 2010.

\bibitem{BLPWII}
H.~H. Bauschke, D.~R. Luke, H.~M. Phan, and X.~Wang.
\newblock Restricted normal cones and sparsity optimization with affine
  constraints.
\newblock arXiv:1205.0320, 2012.

\bibitem{BLPWI}
H.~H. Bauschke, D.~R. Luke, H.~M. Phan, and X.~Wang.
\newblock Restricted normal cones and the method of alternating projections.
\newblock arXiv:1205.0318, 2012.

\bibitem{BeckTeboulle11}
A.~Beck and M.~Teboulle.
\newblock A linearly convergent algorithm for solving a class of
  nonconvex/affine feasibility problems.
\newblock In H.~Bauschke, R.~Burachick, P.~Combettes, V.~Elser, D.~R. Luke, and
  H.~Wolkowicz, editors, {\em Fixed-Point Algorithms for Inverse Problems in
  Science and Engineering}, pages 33--48. Springer, 2011.

\bibitem{CombettesTrussell90}
P.~L. Combettes and H.~J. Trussell.
\newblock Method of successive projections for finding a common point of sets
  in metric spaces.
\newblock {\em J. Opt. Theory and Appl.}, 67(3):487--507, 1990.

\bibitem{Federer59}
H.~Federer.
\newblock Curvature measures.
\newblock {\em Trans. Amer. Math. Soc.}, 93:418--491, 1959.

\bibitem{Hiriart-Urruty12}
J.-B. Hiriart-Urruty.
\newblock When only global optimization matters.
\newblock {\em J. Global Optim.} (2012). doi:10.1007/s10898-011-9826-7

\bibitem{HornJohnson1}
R.~A. Horn and C.~R. Johnson.
\newblock {\em Matrix Analysis}.
\newblock Cambridge University Press, New York, 1985.

\bibitem{Kru05}
A.Y. Kruger.
\newblock Stationarity and regularity of set systems.
\newblock {\em Pac. J. Optim.}, 1(1):101--126, 2005.

\bibitem{Le12}
H.~Y. Le.
\newblock Generalized subdifferentials of the rank function.
\newblock {\em Optimization Letters}, doi: 10.1007/s11590-012-0456-x.

\bibitem{Lewis06}
A.~S. Lewis.
\newblock {Eigenvalues and nonsmooth optimization.}
\newblock {Pardo, Luis M. (ed.) et al., Foundations of computational
  mathematics, Santander 2005. Selected papers based on the presentations at
  the international conference of the Foundations of Computational Mathematics
  (FoCM), Santander, Spain, June 30 -- July 9, 2005. Cambridge: Cambridge
  University Press. London Mathematical Society Lecture Note Series 331,
  208-229 (2006).}.

\bibitem{LewisLukeMalick08}
A.~S. Lewis, D.~R. Luke, and J.~Malick.
\newblock Local linear convergence of alternating and averaged projections.
\newblock {\em Found. Comput. Math.}, 9(4):485--513, 2009.

\bibitem{LewisMalick08}
A.~S. Lewis and J.~Malick.
\newblock Alternating projections on manifolds.
\newblock {\em Math. Oper. Res.}, 33:216--234, 2008.

\bibitem{LewisSendov05}
A.~S. Lewis and H.~S. Sendov.
\newblock {Nonsmooth analysis of singular values. I: Theory.}
\newblock {\em {Set-Valued Anal.}}, {13}({3}):{213--241}.

\bibitem{LewisSendov05b}
A.~S. Lewis and H.~S. Sendov.
\newblock {Nonsmooth analysis of singular values. II: Applications.}
\newblock {\em {Set-Valued Anal.}}, {13}({3}):{243--264}, 2005.

\bibitem{Luke11}
D.~R. Luke.
\newblock Local linear convergence of approximate projections onto regularized
  sets.
\newblock {\em Nonlinear Anal.}, 75:1531--1546, 2012.

\bibitem{Mord84}
B.~S. Mordukhovich.
\newblock Nonsmooth analysis with nonconvex generalized differentials and
  adjoint mappings.
\newblock {\em Dokl. Akad. Nauk BSSR}, 28:976--979, 1984.
\newblock Russian.

\bibitem{Mor06}
B.S. Mordukhovich.
\newblock {\em Variational Analysis and Generalized Differentiation, I: Basic
  Theory; II: Applications}.
\newblock Grundlehren der mathematischen Wissenschaften. Springer-Verlag, New
  York, 2006.

\bibitem{Oszlanyi08}
G.~Oszl\'anyi and A.~S\"ut\'o.
\newblock The charge flipping algorithm.
\newblock {\em Acta Cryst. A}, 64:123--134, 2008.

\bibitem{Pierra76}
G.~Pierra.
\newblock Eclatement de contraintes en parall\`ele pour la minimisation d'une
  forme quadratique.
\newblock {\em Lecture Notes in Computer Science}, 41:200--218, 1976.

\bibitem{PolRockThib00}
R.~A. Poliquin, R.~T. Rockafellar, and L.~Thibault.
\newblock Local differentiability of distance functions.
\newblock {\em Trans. Amer. Math. Soc.}, 352(11):5231--5249, 2000.

\bibitem{RechtFazelParrilo10}
B.~Recht, M.~Fazel, and P.~Parrilo.
\newblock Guaranteed minimum rank solutions of matrix equations via nuclear
  norm minimization.
\newblock {\em SIAM Review}, 52(3):471--501, 2010.

\bibitem{VA}
R.~T. Rockafellar and R.~J. Wets.
\newblock {\em Variational Analysis}.
\newblock Grundlehren der mathematischen Wissenschaften. Springer-Verlag,
  Berlin, 1998.

\end{thebibliography}

\end{document}